\documentclass[12pt,letterpaper]{amsart}

\usepackage{amsmath,amssymb}
\usepackage[margin=2.5cm, top=2.cm, bottom=2.cm, headsep=.6cm, footskip=1cm]{geometry}
\usepackage{enumerate}
\usepackage{hyperref}
\usepackage{lpic}

\theoremstyle{theorem}
\newtheorem*{theorem}{Theorem}
\theoremstyle{theorem}
\newtheorem*{lemma}{Lemma}
\theoremstyle{theorem}
\newtheorem*{proposition}{Proposition}
\theoremstyle{theorem}
\newtheorem*{definition}{Definition}
\theoremstyle{corollary}
\newtheorem*{cor}{Corollary}
\theoremstyle{question}
\newtheorem*{question}{Question}
\theoremstyle{remark}
\newtheorem*{remark}{Remark}

\newcommand{\perim}{\mathrm{perim}}
\newcommand{\area}{\mathrm{area}}
\newcommand{\cP}{\mathcal{P}}
\newcommand{\cG}{\mathcal{G}}
\newcommand{\upT}{\mathrm{T}}

\begin{document}

\author{Jonah Gaster}

\date{September 7, 2020}

\address{Department of Mathematical Sciences, University of Wisconsin-Milwaukee}
\email{gaster@uwm.edu}

\title{A short proof of a conjecture of Aougab-Huang}
\begin{abstract}
In response to Sanki-Vadnere \cite{SV}, we present a short proof of the following theorem: a pair of simple curves on a hyperbolic surface whose complementary regions are disks has length at least half the perimeter of the regular right-angled $(8g-4)$-gon.
\end{abstract}

\maketitle

\section{introduction}
Let $S=S_g$ be an oriented closed surface of genus $g$, and let $\cP=\cP_g$ be the hyperbolic regular right-angled $(8g-4)$-gon. 
A set of curves on $S$ is \emph{filling} if the complementary components are disks.

\begin{theorem}
\label{main thm}
A filling pair of simple geodesics on a hyperbolic surface homeomorphic to $S$ has length at least $\frac12\perim(\cP)$.
\end{theorem}

This theorem was conjectured by Aougab-Huang \cite{AH} in the context of their study of \emph{minimal} filling pairs, i.e.~those for which the complement has one component.
For minimal filling pairs, the above theorem follows directly from an isoperimetric inequality in the hyperbolic plane, due to Bezdek \cite{B}. 
When there are two complementary polygons, one may glue them together along a common side. After erasing two superfluous vertices, the result is an $(8g-4)$-gon, and the same isoperimetric inequality holds \cite[Cor.~4.5]{AH}.

The purpose of this note is to demonstrate that Aougab-Huang's approach generalizes to arbitarily many components:
the complementary pieces can be glued together so that Bezdek's isoperimetric inequality becomes available.
Of course, a difficulty arises, in that the number of sides of the polygon so obtained may have become unwieldy. 
Here one should glue with a bit more care, avoiding the possibility of corners with angle greater than $\pi$.

An alternative technical approach to the above theorem was developed prior to the present paper by Sanki-Vadnere \cite{SV}.
There, the surface $S$ plays a lesser role, and one compares perimeters of the complementary pieces to that of a single regular polygon directly.
Sanki-Vadnere show: let $P_i$ be a polygon with $2n_i$ sides for $i=1,\ldots, r$, and suppose that $P$ is the regular hyperbolic polygon with $\area(P)=\sum \area(P_i)$ and $2m$ sides, where $m+2r=2+\sum n_i$. 
Then, provided $P$ is not acute, we have $\sum \perim(P_i) \ge \perim(P)$.
This somewhat complicated statement implies the Aougab-Huang conjecture: 
The sum of the lengths of the geodesics is at least half of the sum of the perimeters of their complementary components, and by Gauss-Bonnet the polygon obtained above is $P\approx\cP$.

The Sanki-Vadnere result is more general than the above theorem, as it applies to polygons that do not tile a closed surface. 
On the other hand, the approach contained here demonstrates slightly more: if one obtains equality in the above Theorem, then the complement of the geodesics is isometric to $\cP$ (see the Corollary below).

\subsection*{Acknowledgements}
Thanks to Bidyut Sanki and Arya Vadnere, and to Tarik Aougab, Chris Hruska, and Burns Healy for helpful conversations.

\section{A lemma about spanning trees}

Let $G$ be a graph embedded on $S$.
For each vertex $p$ of $G$, the orientation of the tangent space $\upT_p S$ endows the edges incident to $p$ with a cyclic order. 

\begin{definition}
\label{def:spread}
A subgraph $H\subset G$ is \emph{spread} if: for every vertex $p\in H$ and edges $e,e'$ of $H$ at $p$, in the cyclic order at $p$ the edges $e$ and $e'$ are not consecutive.
\end{definition}

\begin{lemma}\label{lem:spread spanning tree}
Let $\alpha,\beta$ be a filling pair of simple closed curves in minimal position on $S$.
If $\alpha$ is nonseparating, then the dual graph to $\alpha\cup \beta$ admits a spread spanning tree. 
If $\alpha$ is separating, the dual graph admits a spread spanning forest with two components.
\end{lemma}

\begin{proof}

Observe that by the assumptions $S$ is homeomorphic to $A/\!\!\sim$, where $A$ is a Euclidean annulus formed by $|\alpha \cap \beta|$ unit squares in a ring, and where $\sim$ is a side-pairing of $A$, so that the core curve of $A$ projects to the homotopy class of $\alpha$ under $A\rightarrow A/\!\!\sim \ \approx S$.
Let $G$ be the 1-skeleton of $A/\!\!\sim$, or, equivalently, the graph dual to $\alpha\cup \beta$. 

The square complex $A/\!\!\sim$ partitions the edges of $G$ into horizontal and vertical.
We suppose that $\alpha$ is horizontal, and let $\Gamma_0$ be the subgraph of $G$ spanned by the horizontal edges. 
Evidently, $\Gamma_0$ spans $G$, since every vertex is incident to a horizontal edge. Moreover, $\Gamma_0$ is spread, since edges alternate between horizontal and vertical at each vertex of $G$.

Now $\Gamma_0$ has either one or two components, according to whether $\alpha$ is nonseparating or separating. Indeed, $\Gamma_0$ is the image of $\partial A\to \partial A/ \!\!\sim$, and the two components of $\partial A$ are connected in the image exactly when $\alpha$ is nonseparating.
We now apply the while-loop:
\begin{align*}
(\ast) \ \ &\text{While $\Gamma_i$ has an embedded loop, let $\Gamma_{i+1}$ be the graph obtained}\\ 
&\text{by deleting an edge that lies in an embedded loop from $\Gamma_i$.}
\end{align*}
This algorithm terminates in a spread spanning tree if $\alpha$ is nonseparating, and it terminates in a spread spanning forest with two components if $\alpha$ is separating.
\end{proof}

\begin{figure}[h]
\includegraphics[width=8cm]{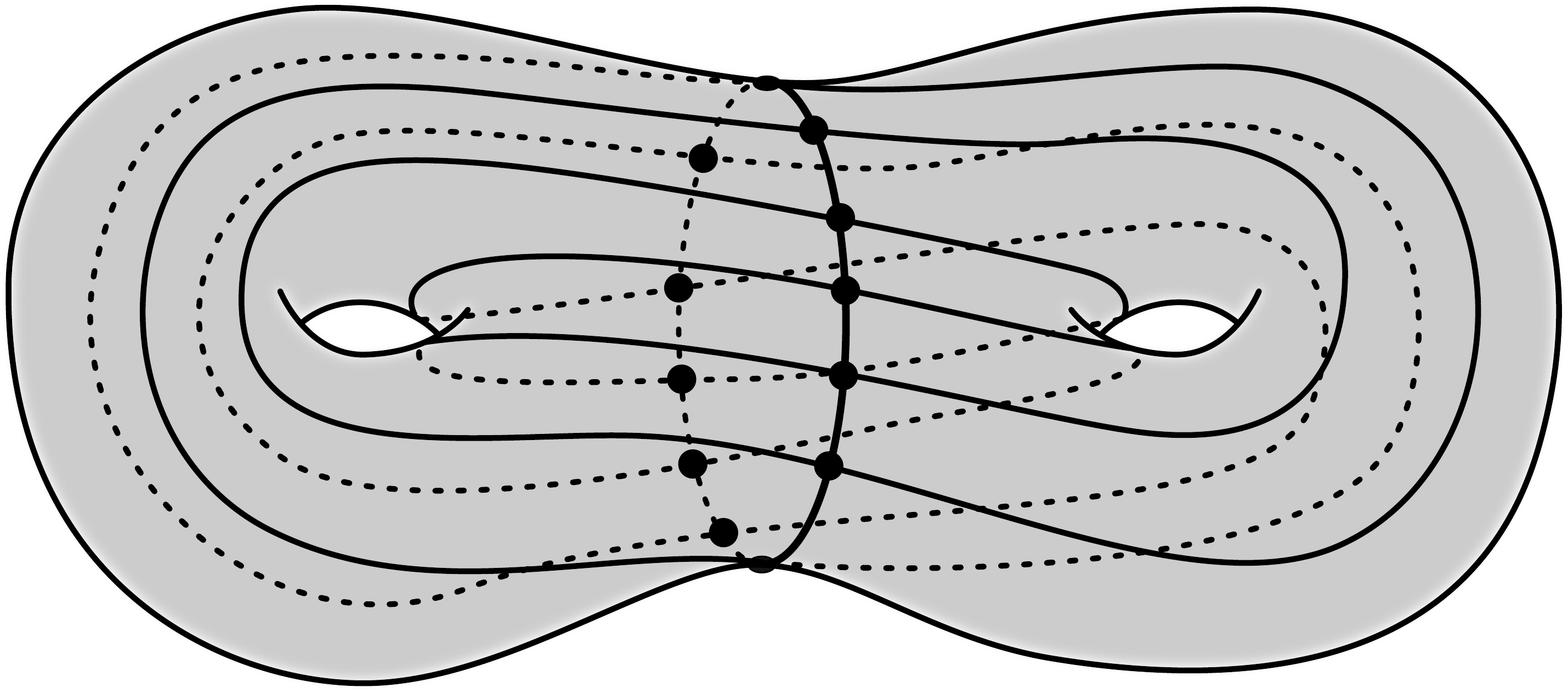}\\
\vspace{.5cm}
\includegraphics[width=\textwidth]{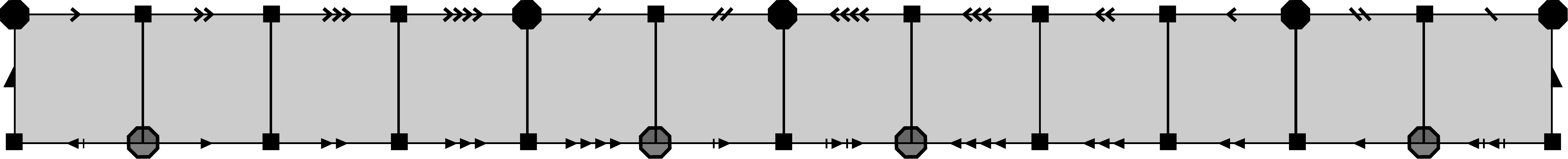}
\caption{A filling pair whose dual graph contains no spread spanning trees.}
\label{fig:counterexample}
\end{figure}

\begin{remark}
The separating / nonseparating dichotomy in this Lemma leads to a dichotomy in the proof of the Main Theorem. 
If one of the two curves is nonseparating, the Aougab-Huang approach goes through unmolested. 
When both curves are separating, more care must be taken.
This dichotomy is not artificial:
Figure~\ref{fig:counterexample} shows a filling pair of separating curves on $S_2$ whose complementary components consist of two octagons and eight squares. 
One can check that there does not exist a spread path between the two octagons.
\end{remark}

\begin{question}
Which filling graphs embedded in $S$ admit spread spanning trees? 
\end{question} 

\section{The proof of the main theorem}

We now mimic the proof of \cite[Cor.~4.5]{AH}, gluing together the complementary polygons to a filling pair of simple geodesics using the above Lemma.
We indicate the perimeter of a polygon $Q$ by $\perim(Q)$ and its number of sides (or, equally, vertices) by $n(Q)$.

\begin{proof}[Proof of Main Theorem]
Let $\alpha,\beta$ be simple geodesics on $X\approx S$, and
let $\cG\subset X$ be the graph induced by $\alpha\cup\beta\subset X$.
The complementary components of $\cG$ determine hyperbolic polygons $P_1,\ldots,P_r$, and
the length of $\cG$ is equal to $\frac12 \sum \perim(P_k)$.

Observe that the sum $\sum n(P_k)$ is two times the number of edges of $\cG$, which is equal to four times the number of vertices. The number of faces is $r$, so by Gauss-Bonnet we find
\begin{equation}
\label{eq:|Q|}
\sum_{k=1}^r n(P_k) = 8g-8+4r~.
\end{equation}

Suppose first that $\alpha$ is nonseparating.
By the lemma, the dual graph to $\cG$ admits a spread spanning tree $T$, which we may regard as embedded in $X$ dual to $\cG$.
Let $\hat Q = \sqcup_k P_k / \sim_T$ be obtained as follows:
for each edge $e$ of $T$ whose endpoints are polygons $P_i$ and $P_j$, we identify the sides of $P_i$ and $P_j$ along their shared side dual to $e$.
As $T$ is a tree, $\hat Q$ is again a polygon.
Moreover, the vertices of $\hat Q$ can be partitioned into \emph{old} vertices, whose $\sim_T$-equivalence class is a singleton, and the complementary \emph{new} vertices. 

Choose a vertex $q\in \cG$.
Because $T$ is spread, the edges of $\cG$ incident to $q$ and dual to edges of $T$ are non-consecutive in the cyclic order of $\cG$ at $q$.
Each $\sim_T$-equivalence class of vertices of $\sqcup_k P_k$ therefore has either one or two elements, and the number of new vertices is exactly $2e(T)$, where $e(T)$ is the number of edges of $T$. 
Moreover, any new vertex of $\hat Q$ must have angle $\pi$, so we may construct a polygon $Q$ by erasing the new vertices of $\hat Q$.

Now it is evident that the number of vertices of $Q$ is equal to the number of old vertices of $\hat Q$, so $n(Q) = n(\hat Q) - 2e(T)$. 
Because $\sim_T$ erases two edges of $\sqcup_kP_k$ for each edge of $T$,
\[
n(\hat Q) = -2e(T) + \sum_{k=1}^r n(P_i)~. 
\]
Together with \eqref{eq:|Q|}, this implies that $n(Q) = -4e(T) + 8g-8+4r$.

As $T$ is spanning, its number of vertices is $r$, and as $T$ is a tree we find $e(T)=r-1$. Hence
\[
n(Q) = -4(r-1) + 8g-8+4r = 8g-4~,
\]
and by \cite{B} we find $\perim(Q) \ge \perim(\cP)$. 
Of course, $\sum \perim(P_k) \ge \perim(Q)$.

Now suppose that $\alpha$ separates $X$ into totally geodesic subsurfaces $X_1$ and $X_2$, of genus $g_1$ and $g_2$ respectively. 
In that case, the Lemma provides the spanning forest $T_1\sqcup T_2\subset \cG$, where $T_1$ and $T_2$ are spread trees. 
The same construction above yields polygons $Q_1$ and $Q_2$ with $\sum \perim(P_k) \ge \perim(Q_1) + \perim(Q_2)$. Moreover, $X_i$ is isometric to a gluing of $Q_i$.

Performing the calculation \eqref{eq:|Q|} for each subsurface, we find $n(Q_i) = 8g_i$. Now let $\hat Q_i$ be a regular $8g_i$-gon with area $\pi(4g_i-2)$, so that by Bezdek we find $\perim(Q_i)\ge \perim(\hat Q_i)$. Observe that $\hat Q_i$ is necessarily right-angled. Indeed, the common angle of $\hat Q_i$ is given by
\[
\frac \pi{8 g_i} (8g_i-2-(4g_i-2)) = \frac \pi 2~.
\]
The following comparison now completes the proof:
\begin{proposition}
\label{prop: right-angled perims}
Suppose that $R_1$, $R_2$, and $R$ are regular right-angled polygons with $n(R_i)=n_i$, $n(R)=m$, and suppose that $n_1+n_2=m+4$. Then $\perim(R_1)+\perim(R_2) > \perim(R)$.\qedhere
\end{proposition}
\end{proof}

Observe that we may conclude as well: if $r>1$, then $\sum\perim(P_k)>\perim(\cP)$. 
Therefore,

\begin{cor}
\label{cor:uniqueness}
With the setup of the Theorem, if we find equality in the conclusion, then the filling pair is minimal and $X$ is obtained as a gluing of $\cP$. 
\end{cor}

It remains to prove the above proposition.
We emphasize that the right-angled hypothesis makes this statement far simpler than the involved calculations of \cite{SV}.

\begin{proof}[Proof of Proposition]
As $R_i$ is a right-angled hyperbolic polygon, we have $n_i\ge5$, so the constraint $n_1+n_2=m+4$ implies that $n_1,n_2\in\{5,\ldots,m-1\}$. 
We first show that, for fixed $m$, the sum $\perim(R_1)+\perim(R_2)$ is minimized for $\{n_1,n_2\}=\{5,m-1\}$. 

One may use hyperbolic trigonometry to calculate the perimeter of a regular polygon (see \cite[p.~97]{Ratcliffe}). 
Using the right-angled assumption we find that $\perim(R_1) = f(n_1)$, $\perim(R_2)=f(n_2)$, and $\perim(R)=f(m)$, where 
\[
f(x)= 2x \cosh^{-1}\left(\sqrt 2 \cos\left(\frac{\pi}{x}\right) \right)~. 
\]
It is straightforward to compute
\begin{align*}
f'(x) = 2 \cosh^{-1}\left(\sqrt 2 \cos\left(\frac{\pi}{x}\right) \right) + 
\frac{2\pi \sqrt2 \sin\left( \frac \pi x\right) }{x\sqrt{\cos\left(\frac{2\pi}x\right)}} ~, \text{ and } \ 
f''(x) = -\frac{2\pi^2\sqrt{2} \ \cos\left( \frac \pi x\right)}{x^3 \ \sqrt{\cos^3 \left( \frac{2\pi}x\right)}} ~.
\end{align*}

Because $f$ is concave (i.e.~$f''(x)<0$), the function $f(x)+f(C-x)$ is concave as well, for any $x$ and constant $C$ so that the sum is defined.
Therefore, as a function of $n_1\in(4,m)$, $\perim(R_1)+\perim(R_2) = f(n_1)+f(m+4-n_1)$ is concave. 
As $n_1$ is an integer, we must have $\perim(R_1)+\perim(R_2)\ge f(5)+f(m-1)$. 
Observe that the desired inequality now follows from $f(5)+f(m-1)> f(m)$.
Concavity of $f$ implies that 
\[
f(m) - f(m-1) < (m-(m-1)) \cdot f'(m)=f'(m) \le f'(5)~.
\]
It remains to show that $f'(5)\le f(5)$.
As $\cos\left(\frac\pi 5\right) = \frac\Phi 2$, where $\Phi= \frac{1+\sqrt 5}{2}$, we compute
\begin{align*}
f'(5) = 2\cosh^{-1} 
\left( \frac1{\sqrt2} \Phi\right) + \frac{2\pi}{5}\sqrt{\Phi+\frac1\Phi}~, \text{ and }  \
f(5) = 10\cosh^{-1} 
\left( \frac1{\sqrt2} \Phi\right) ~.
\end{align*}
As $\pi<4$, we find that $f'(5)< f(5)$ is implied by $\frac15 \sqrt{\Phi+\frac1\Phi} \le \cosh^{-1}\left( \frac1{\sqrt2} \Phi\right)$.

While one can check that $\frac15 \sqrt{\Phi+\frac1\Phi} \approx .299$ and $\cosh^{-1}\left( \frac1{\sqrt2} \Phi\right) \approx .531$ with a calculator, in fact this can be checked by hand.

Exploiting $\Phi^2=1+\Phi$, $\frac1\Phi = \Phi-1$, and $\cosh^{-1}x=\log(x+\sqrt{x^2-1})$, one finds:
\begin{align*}
\cosh^{-1}\left( \frac1{\sqrt2} \Phi \right) &= 
\frac12\log\left( \Phi+\sqrt\Phi\right) > \frac12 >\frac15 \sqrt{\Phi+\frac1\Phi}~, 
\end{align*}
where on the last line we've used the elementary estimates $\Phi+\sqrt\Phi >e$ and $\Phi+\frac1\Phi <3$. \qedhere

\end{proof}

\end{document}